\documentclass{amsart}

\usepackage[top = 1in, left = 1.25in, right = 1.5in, bottom = 1in]{geometry}

\usepackage{amsrefs}
\usepackage{mathrsfs}
\usepackage{verbatim}

\def\be{\begin{equation}}
\def\en{\end{equation}}

\newcommand{\htop}{h_{\text{\normalfont top}}}

\newtheorem{theorem}{Theorem}[section] 

\newtheorem{proposition}[theorem]{Proposition}

\theoremstyle{definition}

\newtheorem{example}[theorem]{Example}
\newtheorem*{ack*}{Acknowledgments}

\theoremstyle{remark}

\numberwithin{equation}{section}

\title{Tree shift topological entropy}

\author{Karl Petersen}
\address{Department of Mathematics,
	CB 3250 Phillips Hall,
	University of North Carolina,
	Chapel Hill, NC 27599 USA}
\email{petersen@math.unc.edu}

\author{Ibrahim Salama}
\address{School of Business, North Carolina Central University, 
	Durham, NC 27707 USA}
\email{isalama@nccu.edu}

\date{\today}

\begin{document}
	
	
	\begin{abstract}
	We give a definition of topological entropy for tree shifts, prove that the limit in the definition exists, and show that it dominates the topological entropy of the associated one-dimensional shift of finite type when the labeling of the tree shares the same restrictions. 
	\end{abstract}
	
\maketitle
\section{Introduction}
Tree shifts were introduced by Aubrun and B\'{e}al \cites{AB1,AB2,AB3,AB4,AB5} as interesting objects of study, since they are more complicated than one-dimensional subshifts while preserving some directionality, but perhaps not so hard to analyze as multidimensional subshifts. 
They have been studied further by Ban and Chang \cites{BC1,BC2,BC3,BC4,BC5}. 
We consider here the complexity of tree shifts and labeled trees in general, especially two variations of topological entropy.
We define the topological entropy of a tree shift in a different manner than Ban and Chang \cite{BC3}, prove that the limit in the definition exists, and consider ways to estimate it. 
In particular, if the tree shift consists of all trees labeled by a finite alphabet subject to adjacency restrictions given by a $0,1$ matrix that also defines a one-dimensional shift of finite type (SFT), we prove that the entropy of the tree shift is bounded below by that of the SFT.

 Although much of what we say extends to general trees, we focus on labelings of the standard infinite dyadic tree, which 
 corresponds to the set of all finite words on a two-element alphabet, for us 
 $\Sigma= \{ L,R\}$. 
As usual, $\Sigma^0$ is the empty word $\epsilon$, for $n \geq 1$ we denote by $\Sigma^n$ the set of all words of length $n$ on the alphabet $\Sigma$, and $\Sigma^*=\cup_{n \geq 0}\Sigma^n$.
A word $w \in \Sigma^*$ corresponds uniquely to a path in the tree from the root and to the vertex which is at the end of that path. We denote by $|w|$ the length of the word $w$.  
There are one-to-one correspondences between elements of $\Sigma^*$, nodes of the tree, and finite paths starting at the root. 

Let $A=\{ 0,1\}$ be a labeling alphabet. Then a {\em labeled tree} is a function $\tau : \Sigma^* \to A$. For each $w \in \Sigma^*$, $\tau (w)$ is thought of as the label attached to the node at the end of the path determined by $w$. 
The two shifts on $\Sigma^*$  are defined by $\sigma_i (w) = iw, i=L,R$. 
For $w=w_0 w_1 \dots w_{n-1} \in \Sigma^*$, define $\sigma_w= \sigma_{n-1} \dots \sigma_1 \sigma_0$ (note the reverse order, since $\sigma_w$ is a left action on words which are ordered left to right).
On a labeled tree $\tau$, define $(\sigma_i \tau )(w) = \tau(iw), i=L,R$. 

For each $n \geq 0$ let $\Delta_n = \cup_{0 \leq i \leq n} \Sigma^i$ denote the initial $n$-subtree of the dyadic tree. 
Then $\Delta_n$ has $n+1$ ``levels" and $2^{n+1}-1$ nodes. 
For any $x \in \Sigma^*$, the {\em shift to $x$} of $\Delta_n$ is the $n$-subtree $x \Delta_n = \sigma_x \Delta_n$.
An {\em $n$-block} is a function $B: \Delta_n \to A$, i.e., a labeling of the nodes of $\Delta_n$, or a {\em configuration} on $\Delta_n$.
We may write a $1$-block $B$ for which $B(\epsilon)=a, B(L)=b,B(R)=c$ as $a_b^c$. 
We say that an $n$-block $B$ {\em appears} in a labeled tree $\tau$ if there is a node $x \in \Sigma^*$ such that $\tau(xw)=B(w)$ for all $w \in \Delta_n$. 
Denote by $p_\tau$ the {\em complexity function} of the labeled tree $\tau$: for each $n \geq 0$, $p_\tau (n)$ is the number of distinct $n$-blocks that appear in $\tau$. 
A {\em tree shift} $X$ is the set of all labeled trees which omit all of a certain set (possibly infinite) of forbidden blocks. These are exactly the closed shift-invariant subsets of the full tree shift space $T(A)=A^{\Sigma^*}$. 
We deal here only with {\em transitive} tree shifts $X$, those for which there exists $\tau \in X$ such that every block that appears in $X$ appears in $\tau$.

The {\em complexity function} $p_\omega (n)$ of a sequence $\omega$ on a finite alphabet gives for each $n$ the number of distinct blocks (or words) of length $n$ found in the sequence. 
Hedlund and Morse \cite{HM} showed that for a one-sided sequence the following statements are equivalent: (1) there is an $n$ such that $p_\omega(n) \leq n$; 
(2) there is a $k$ such that $p_\omega(k+1)=p_\omega(k)$;
(3) $\omega$ is eventually periodic;
(4) $p_\omega$ is bounded. 
Moreover, Coven and Hedlund \cite{CovenHedlund1973} identified the two-sided sequences that have minimal unbounded complexity as exactly the Sturmian sequences together with certain concatenations of one-sided periodic sequences. 
In 1997 M. Nivat conjectured that an analogue of these statements might hold in higher dimensions, specifically for labelings of the integer lattice in $\mathbb Z^2$ by a finite alphabet; 
see for example \cites{Epifanio, Cyr-Kra, KPSalta} for the precise statement, some positive progress, and references. 
Extensions of these statements to labeled trees were provided in \cites{Carpi2001, Berstel2010}. See also \cites{GastGaujal2010,KimLim2015}.

The {\em complexity function} $p(n)$ of a tree shift gives for each $n \geq 0$ the number of $n$-blocks among all labeled trees in the tree shift. 
Recall that we deal only with transitive tree shifts; thus if $\tau \in X$ has the property that every block that appears in $X$ also appears in $\tau$, then $p(n)=p_\tau (n)$ for all $n$.
The {\em entropy} of the tree shift is defined to be 
\be
h=\limsup_{n \to \infty} \frac{\log p(n)}{2^{n+1}-1},
\en
the exponential growth rate of the number of different labelings of shifts of $\Delta_n$ in $\tau$ divided by the number of sites for the labels. 
(We could equivalently divide by just $2^{n+1}$ instead of $2^{n+1}-1$.)
Ban and Chang use a different definition,
\be
h_2=\limsup_{n \to \infty} \frac{\log\log p(n)}{n}.
\en
They show that for labelings of dyadic trees consistent with 1-step finite type restrictions on adjacent symbols, $h_2$ is always either $0$ or $\log 2$ \cite{BC3}. 
We will see below that the quantity $h$ behaves quite differently.

\section{Existence of the limit}

Since factor maps between tree shifts are given by sliding block codes, the complexity function is nonincreasing under factor maps, and hence both versions of entropy are invariants of topological conjugacy of tree shifts. 
Ban and Chang show that for many tree shifts $h_2$ exists as a limit. For our $h$ we have the following.

\begin{theorem}\label{prop:limit}
	The limit $	h=\lim_{n \to \infty} {\log p(n)}/{2^{n+1}}$
	exists. In fact, for each labeled tree $\tau$ the limit $h(\tau)=\lim_{n \to \infty} {\log p_\tau(n)}/{2^{n+1}}$ exists.
\end{theorem}
\begin{proof}
	Let $\tau$ be a labeled tree as above. We follow the strategy of the standard proof using subadditivity, adapting as necessary to the tree situation.
	We have, for $m,n \geq 0$,
	\be
	p_\tau(m+n) \leq p_\tau(m) p_\tau(n)^{2^m}.
	\en
	This is because $\Delta_m$ has $2^m$ terminal nodes, and $\Delta_{m+n}$ is obtained
	by attaching a shift of $\Delta_n$ to each one of these terminal nodes. 
	Thus any configuration on a shift of $\Delta_{m+n}$ in $\tau$ consists of a configuration on a shift of $\Delta_m$ with $2^m$ configurations on shifts of $\Delta_n$ attached to the terminal nodes.
	
	 For any (positive) integer $k$,
	\be
	p_\tau(km) \leq p_\tau(m)^{(2^{km}-1)/(2^m-1)}.
	\en
	At the root of any shift of $\Delta_{km}$, we have one $\Delta_m$, with
	$2^m$ terminal nodes. Attached to its terminal nodes we have $2^m$ shifts of $\Delta_m$, 
	bringing the total number of shifts of $\Delta_m$ so far to $1+2^m$. 
	Next, we have
	$2^{2m}$ terminal nodes, with a shift of $\Delta_m$ attached to each one, bringing
	the total so far to $1+2^m+2^{2m}$. Continuing this way, at the last level we
	have $2^{(k-1)m}$ terminal nodes, with a shift of $\Delta_m$ attached to each one,
	bringing the total number of shifts of $\Delta_m$ to $1+2^m+2^{2m}+\cdots+2^{(k-1)m}$. 
	The labeling in $\tau$ of the shift of $\Delta_{km}$ gives labelings in $\tau$ of each of these shifts of $\Delta_m$ (but these latter labelings might not all be independent).
	
	We show now that $\lim {\log p_\tau(n)}/{2^{n+1}}$ exists. 
    Let 
    \be
    \alpha=\liminf \frac{\log p_\tau(n)}{2^{n+1}}=\liminf\frac{\log p_\tau(n)}{2^{n+1}-2}. 
    \en
	 Let 
	$\epsilon >0$ and choose an integer $r \geq 1$ such that 
	 	${\log p_\tau(r)}/{(2^{r+1}-2)}<\alpha+{\epsilon}$.
	Given an integer $n$,  write $n=i+kr$, with $0\leq i <r$. 
	Then
	\be
	\begin{aligned}
	p_\tau(n)& = p_\tau(i+kr) \leq p_\tau(i) p_\tau(kr)^{2^i} \\\
	& \leq p_\tau(i)  [p_\tau(r)^{(2^{kr}-1)/(2^r-1)}]^{2^i} \\
	& = p_\tau(i) p_\tau(r)^{2^i(2^{kr}-1)/(2^r-1)}.
	\end{aligned}
	\en
Now if $n$ is large enough,
\be
\begin{aligned}
	\alpha - \epsilon &\leq \frac{\log p_\tau(n)}{2^{n+1}}\\
	 & \leq  \frac{ \log p_\tau(i)}{2^{n+1}}
	+\frac{1}{2^{i+kr+1}}\frac{2^i(2^{kr}-1)}{2^r-1} \log p_\tau(r)  \\
	& = \frac{\log p_\tau(i)}{2^{n+1}}+
	\frac{1-2^{-kr}}{2^{r+1}-2} \log p_\tau(r) \\
	& <  \frac{\log p_\tau(i)}{2^{n+1}}+\frac{\log p_\tau(r)}{2^{r+1}-2} \\
	& <   \frac{\log p_\tau(i)}{2^{n+1}}  + \alpha + \epsilon\\
	& < \alpha+2\epsilon.  
	\end{aligned}
\en
	\end{proof}

\section{Entropy estimates for tree shifts determined by one-dimensional SFT's}
Consider now a dyadic tree shift with vertices labeled from a finite alphabet $A=\{a_1,\dots,a_d\}$ with $1$-step finite type restrictions given by a $0,1$ matrix $M$ indexed by the elements of $A$:  adjacent nodes  in the tree are allowed to have labels $i$ for the first (closer to the root) and $j$ for the second if and only if $M_{ij}=1$.
We wish to find, or at least estimate, the entropy $h$ of the tree shift.

\begin{proposition}
	Suppose that the matrix $M$ is irreducible, and for each $i=1,\dots,d$ and $n \geq 0$ denote by $x_i(n)$ the number of labelings of $\Delta_n$ which are consistent with the adjacency matrix $M$ and have the symbol $a_i$ at the root.
	Then for each $i$ and $j$,
	\be\label{eq:limsups}
	\limsup_{n\to \infty} \frac{\log x_i(n)}{2^{n+1}} =\limsup_{n \to \infty} \frac{\log x_j(n)}{2^{n+1}} = h.
	\en
	If $M$ is primitive (irreducible and aperiodic, so that the corresponding one-dimensional shift of finite type is topologically mixing), then for each $i$ and $j, \lim_{n\to \infty} {\log x_i(n)}/{2^{n+1}}$ exists, and
		\be
		\lim_{n\to \infty} \frac{\log x_i(n)}{2^{n+1}} =\lim_{n \to \infty} \frac{\log x_j(n)}{2^{n+1}} = h.
		\en
\end{proposition}
\begin{proof}
	By irreducibility, given $i,j \in \{1,\dots,d\}$, there is a word in the subshift $\Sigma_M$ of some length $k(i,j)$ that begins with $a_i$ and ends with $a_j$.
	By labeling each path from the root to a vertex at level $k(i,j)$ with this same word, we can find a labeling consistent with $M$ that has $a_i$ at the root and $a_j$ at every vertex at level $k(i,j)$. 
	Now for any $n \geq 0$ the shifts to these vertices at level $k(i,j)$ of $\Delta_n$ can be labeled independently in $x_j(n)$ ways. 
	This implies that 
	\be
	x_i(n+k(i,j)) \geq x_j(n)^{2^{k(i,j)}},
	\en
	and Equation (\ref{eq:limsups}) follows.
	
	Suppose now that $M$ is primitive. Then it is possible to find a single $k$ such that for all $i,j$ there is a word in $\Sigma_M$ of length $k$ that begins with $a_i$ and ends with $a_j$. Then for all $i,j$,
	\be
	x_i(n+k)^{1/2^k} \geq x_j(n),
	\en
	so that for all $i$
	\be
	d x_i(n+k)^{1/2^k} \geq p_\tau (n),
	\en
	and
	\be
	h=\lim_{n \to \infty} \frac{\log p_\tau(n)}{2^{n+1}} \geq \limsup_{n \to \infty} \frac{\log x_i(n)}{2^{n+1}} 
		\geq \liminf_{n \to \infty} \frac{\log x_i(n+k)}{2^{n+k+1}} \geq \liminf_{n \to \infty} \frac{\log p_\tau(n)}{2^{n+1}} = h.
		\en
	\end{proof}
	
\begin{example}
	Consider the matrix $M$ whose rows are $0100, 0010, 0101, 1000$. 
	We will show that for the corresponding tree shift we have 
	\be
	h_i = \liminf_{n \to \infty} \frac{\log x_i(n)}{2^{n+1}} \leq \frac{h}{2} < h \quad\text{ for all } i,
	\en
	so that the limits $\log x_i(n)/2^{n+1}$ do not exist.
	
	The recurrence equations for the $x_i(n)$ are
	\be
	\begin{aligned}
		x_1(n+1)&=x_2(n)^2\\
		x_2(n+1)&=x_3(n)^2\\
		x_3(n+1)&=(x_2(n)+x_4(n))^2\\
		x_4(n+1)&=x_1(n)^2,
		\end{aligned}
	\en
	with $x(0)=(1,1,1,1)$.
	One may compute by hand that $x(1)=(1,1,4,1), x(2)=(1,16,4,1),x(3)=(256,16,289,1)$.
	Using induction, for all $n \geq 1$,
	\be
	\begin{aligned}
	(x_1(2n),x_3(2n))&=(x_1(2n-1),x_3(2n-1)),\\
	(x_2(2n-1),x_4(2n-1))&=(x_2(2n-2),x_4(2n-2)).
	\end{aligned}
	\en
	Thus
	\be
	\limsup_{n \to \infty} \frac{\log x_i(n)}{2^{n+1}} =h \quad\text{ for all } i
	\en
 must occur along even $n$ for $x_2$ and $x_4$ and along odd $n$ for $x_1$ and $x_3$.
 	Then we have, for example,
\be
\limsup _{n \to \infty} \frac{\log x_1(2n)}{2^{2n+1}} = \limsup_{n \to \infty} \frac{\log x_1(2n-1)}{2^{2n+1}}= \frac{1}{2} \limsup_{n \to \infty} \frac{\log x_1(2n-1)}{2^{(2n-1)+1}}=\frac{h}{2}.
\en
\end{example}

The following comparison of the entropy of a one-dimensional shift of finite type with the tree shift that it determines is intuitively plausible, but not trivial to prove. A possible extension is mentioned below in Question 8. Note for comparison that numerical evidence \cite{GamarnikKatz2009} indicates that for the golden mean SFT's on the integer lattices $\mathbb Z^d$ (also called the hard square or hard core models) the topological entropy seems to {\em decrease} with $d=1,2,3,4$. 

	\begin{theorem} \label{thm:htop}
	Let $M$ be an irreducible $d$-dimensional $0,1$ matrix, $\Sigma_M$ the corresponding shift of finite type, and $X_M$ the corresponding tree shift, labeled by elements of the alphabet $A$, with $|A|=d$, subject to the adjacency restrictions given by $M$.
	Then the topological entropy of $\Sigma_M$ is less than or equal to the topological entropy of $X_M$: $\htop(\Sigma_M) \leq  h$.
\end{theorem}
\begin{proof}  
	Let $v$ denote the positive Perron-Frobenius left eigenvector of $M$ normalized so that $\sum v_i=1$, and let $\lambda>0$ denote the maximum eigenvalue of $M$. 
	For each $n=0,1,\dots$ and $i=1,\dots,|A|$ denote by $x(n)=(x_i{(n)}), i=1,\dots ,|A|$, the vector that gives for each symbol $i \in A$ the number of trees of height $n$ labeled according to the transitions allowed by $M$ that have the symbol $i$ at the root. 
	Considering the symbols that can follow each symbol $i$ in the last row of a labeling of $\Delta_n$, and that they can be assigned in independent pairs to the nodes below, shows that 
	these vectors satisfy the recurrence 
	\be
	x_i(0)=1, \quad x_i(n+1)=(Mx(n))_i^2 \text{ for all } i=1,\dots,d, \text{ all } n \geq 0.
	\en
	(Cf. the nonlinear recurrences in \cite{BC3}.)
	Denote by $1$ the vector $(1,1,\dots,1) \in \mathbb R^d$. 
	We claim that 
	\be
	x(n) \cdot v \geq \lambda^{2^{n+1}-2}v \cdot 1 \quad\text{ for all } n\geq 0.
	\en
	Since all entries of $v$ are positive, $x(n) \cdot v$ and $x(n) \cdot 1$ grow at the same superexponential rate, so
	the result follows.
	
	For $n=0$ we have
	\be
	x(0) \cdot v =\sum_i v_i =v \cdot 1.
	\en
	
	Assuming that the inequality holds at stage $n$ and using the inequality $\mathbb E(X^2) \geq [\mathbb E(X)]^2$ on the  random variable $X_i = [Mx(n)]_i$ with discrete probabilities $v_i$, 
	we have
	\be
	\begin{aligned}
		&\sum_ix_i(n+1) v_i=\sum_i(Mx(n))_i^2 v_i \geq \left[\sum_i Mx(n)_i v_i \right]^2\\
		&=\left[\sum_i x(n)_i (vM)_i \right]^2 = \left[\sum_i x(n)_i  \lambda v_i \right]^2 = \left[\lambda x(n) \cdot v \right]^2\\
		&\geq \left[\lambda^{2^{n+1}-2} \lambda v \cdot 1\right]^2
		=  \lambda^{2^{n+2}-2} v \cdot 1.
	\end{aligned}
	\en
\end{proof}

\begin{proposition}\label{prop:upper}
		Let $M$ be an irreducible $d$-dimensional $0,1$ matrix, $\Sigma_M$ the corresponding shift of finite type, and $X_M$ the corresponding tree shift, labeled by elements of the alphabet $A$, with $|A|=d$, subject to the adjacency restrictions given by $M$.
		Let $r$ be a positive right eigenvector of $M$, $\lambda$ the maximum eigenvalue of $M$, $r_{\max}=\max\{r_i\}, r_{\min}=\min\{ r_i\}$, and $c=r_{\max}/r_{\min}$.
		Then the entropy $h$ of the tree shift $X_M$ satisfies the estimate
		\be
		h \leq U = \frac{1}{2} \log c + \log \lambda.
	\en
	\end{proposition}
\begin{proof}
	We use the following estimate for the number $N_n(a)$  of $n$-blocks that can follow any symbol $a \in A$ in the SFT $\Sigma_M$ (see \cite[p. 107]{LM}):
	\be
	N_n(a) \leq c \lambda^n.
	\en
	To label the nodes of $\Delta_n$ with symbols from the alphabet $A$ consistent with the adjacency restrictions prescribed by $M$, 
	we first label the nodes down the left edge $LL \dots L$; there are no more than $d c \lambda^n$ ways to do this. 
	Then we label the paths below, taking into account nodes above that have already been labeled. 
		Thus there are less than or equal to $d c \lambda^2 c \lambda = d c^2 \lambda^3$ ways to label $\Delta_1$, 
	no more than $d c \lambda^3 c \lambda c \lambda^2 c \lambda = d c^4 \lambda^7$ to label $\Delta_2$, etc. 
	Using induction, we estimate that there are no more than 
	\be
	d c^{2^n} \lambda^{2^{n+1}-2}
	\en
	ways to label $\Delta_n$ that observe the adjacency restrictions,
	and this gives
	\be
	h \leq U =\frac{1}{2} \log c + \log \lambda.
	\en	
	\end{proof}

\section{Recurrence and numerics for the golden mean}
We compute the entropies of the one-dimensional SFT and of the tree shift numerically for the golden mean restrictions, for which no two adjacent nodes are allowed to have the same label $1$, by means of a recurrence equation in just one variable. 
This approach shows explicitly how the characteristic polynomial of the adjacency matrix of the one-dimensional SFT enters and may be helpful in the analysis of other examples.

Direct observation gives $p(0)=2$ and $p(1)=5$. With some more effort we find that $p(2)=41$ and $p(3)=2306$. Consulting the Online Encyclopedia of Integer Sequences \cite{OEIS} finds Sequence A076725, the number of independent sets (no two nodes in the set are adjacent) in a dyadic tree with $2^n - 1$ nodes.

Referencing Jonathan S. Braunhut, OEIS states that the sequence satisfies the recursion 
\be
p(n+1)=p(n)^2 + p(n-1)^4.
\en
This may be verified as follows. 
A labeled shift of $\Delta_{n+1}$ with no $11$ either begins with $0$ at its root, in which case it can be followed by two independently labeled shifts of $\Delta_n$'s with no $11$ below, 
or else it begins with $1_0^0$ at the root followed  by four independently labeled shifts of $\Delta_{n-1}$'s below. 

OEIS gives numerical evidence that
\be
p(n) \asymp bc^{2^{n+2}}, \qquad\text{with } c \approx 1.28975 \quad\text {and } b \approx 0.6823278.
\en
Then we find that 
\be
h=2\log c \approx 0.509 \qquad\text{ and } \qquad h_2 = \log 2 \approx 0.693.
\en

We verify that the limit for $h$ exists, by using the recurrence equation $p(n+1)=p(n)^2 + p(n-1)^4.$

\begin{proposition}
	Suppose that $p(0)=2, p(1)=5$, and for $n \geq 2$ we have $p(n+1)=p(n)^2 + p(n-1)^4. $
	Let  $q(n)=p(n)/[p(n-1)]^2$ for $n \geq 2$. 
	Then $\lim_{n \to \infty} q(n)$ exists (it is the real root of $x = 1+1/x^2$, which is approximately 1.46557), 
	and therefore $\lim \log p(n)/2^{n+1}$ exists (it is approximately $0.509$).
\end{proposition}
\begin{proof}
	We have $q(n)=1 + 1/q(n-1)^2, q(2)=5/4=1.25, q(3)=41/25 \approx 1.64,q(4)=625/1681 \approx 1.3718$.
	Using the recurrence,
	\be
	q_n-q_{n-1}=\frac{q_{n-2}^2 - q_{n-1}^2}{q_{n-1}^2q_{n-2}^2},
	\en
	so that $q_n - q_{n-1}$ is alternating in sign. Moreover,
	\be
	|q_n - q_{n-1}| = |q_{n-2} - q_{n - 1}| \frac{q_{n-2} + q_{n-1}}{q_{n-1}^2 q_{n-2}^2},
	\en
	and the factor ${q_{n-2} + q_{n-1}}/{q_{n-1}^2 q_{n-2}^2}$ is bounded by $0.891657 < 1$, since, viewing the first few values of the sequence $q(n)$ and using the fact that it is alternately increasing and decreasing  
	we have both of $q_{n-2}, q_{n-1} \geq 1.25$ and at least one of them no less than $1.3718$. 
	Numerical evidence from OEIS shows that 
	\be
	q_n = \frac{p_n}{p_{n-1}^2} \to 1.46557\dots .
	\en
	
	We show now that this implies the existence of the limit $\lim \log p_n/2^{n+1}$.
	From above, since $q_n \to 1.46557\dots$, we have $\log p_n -2 \log p_{n-1} \to \log 1.46557\dots = a=0.382244\dots$.
	Thus
	\be
	\left|	\log p_n - 2 \log p_{n-1} - a\right| = \epsilon_n \to 0,
	\en
	and hence
	\be
		\left|\left(\frac{\log p_n}{2^{n+1}}+\frac{a}{2^{n+1}}\right) - \left(\frac{\log p_{n-1}}{2^n} +\frac{a}{2^{n}}\right)\right| =
	\left|\frac{\log p_n}{2^{n+1}} - \frac{\log p_{n-1}}{2^n} -\frac{a}{2^{n+1}}\right| 
	< \frac{\epsilon_n}{2^{n+1}}.
	\en
	Therefore the sequence 
	\be
	\left(\frac{\log p_n}{2^{n+1}}+\frac{a}{2^{n+1}}\right)
	\en
	converges, and hence so does the sequence $\log p_n/2^{n+1}$.
\end{proof}

\begin{proposition}
	For each $n \geq 0$ denote by $A_n$ the number of dyadic trees labeled by $A=\{0,1\}$ in such a way that $0$ labels the root and no two adjacent nodes both have the label $1$. (Thus  $A_{n+1}=(A_n+A_{n-1}^2)^2$ for $n \geq 1$ and $(A_n)=(1,4,25, 1681, 5317636,\dots)$.) 
	Denote by $\gamma$ the golden mean $(1+\sqrt{5})/2$, so that $\gamma+1=\gamma^2$ and the one-dimensional golden mean shift of finite type $\Sigma_M$  has entropy $\log \gamma$.
	Then 
	\be
	A_n \geq \gamma^{2^{n+1}-1} \quad\text{ for } n\geq 4,
	\en
	so that 
	\be
	h = \lim_{n \to \infty} \frac{\log A_n}{2^{n+1}-1} \geq \log \gamma = \htop(\Sigma_M).
	\en 
\end{proposition}
\begin{proof}
	Direct calculation verifies the inequality for $n=4,5$. Using induction,
	\be
	\begin{aligned}
		A_{n+1}&=(A_n+A_{n-1}^2)^2 \geq [\gamma^{2^{n+1}-1}+(\gamma^{2^n-1})^2]^2\\
		&=(\gamma^{2^{n+1}-1}+\gamma^{2^{n+1}-2})^2 = \gamma^{2^{n+2}-2}+2\gamma^{2^{n+2}-3}+\gamma^{2^{n+2}-4}\\
		&=\gamma^{2^{n+2}-4}(\gamma^2+2\gamma+1)=\gamma^{2^{n+2}-4}(\gamma+1)^2=\gamma^{2^{n+2}-4}\gamma^4\\
		&=\gamma^{2^{n+2}}\geq \gamma^{2^{n+2}-1}.
	\end{aligned}
	\en
\end{proof}

\section{Examples and questions}

The following table shows {\em Mathematica} computations for one $2 \times 2$ and fourteen $3 \times 3$ one-dimensional shifts of finite type, with entropies $\htop$, determining labelings of the dyadic tree. 
The entropy $h$ of the tree is estimated by recurrence up to $n=15$. 
If the transition matrix $M$ has all row sums equal to $s$, then $h=\htop(\Sigma_M)=U=\log s$.
For an upper estimate we use $U=(\log c)/2+\log\lambda$ from Proposition \ref{prop:upper}.

\medskip
\begin{table}[h]
\centering
\begin{tabular}{l l l l l l l l l l l}
	Name  &Matrix$=M$  &$\htop(\Sigma_M)$   &$h$ (est)                    &$U$   \\
	\hline
	$\Gamma$  &$11,10$              &$.481$   &$.509$                &.721  \\  
	$X_0$         &$010,101,101$    &$.481$   &$.509$                 &.722 \\
	$X_1$         &$110,001,110$     &$.481$   &$.509$                   &.722  \\
	$X_2$        &$011,101,100$    &$.481$     &$.509$                    &.722 \\
	$X_3$        &$011,111,101$     &$.81$     &$.846$                   &1.104  \\
	$X_4$         &$111,110,100$    &$.81$    &$.846$                  &1.214 \\
	$X_5$     &$110,011,101$      &$.693$  &$.693$                  &.693  \\
	$X_6$       &$011,101,110$    &$.693$   &$.693$                 &.693  \\
	$X_7$    &$110,001,111$       &$.693$    &$.768$                &1.04  \\
	$X_8$      &$110,011,110$     &$.693$     &$.693$              &.693  \\
	$X_9$      &$011,101,101$      &$.693$     &$.693$             &.693  \\
	$X_{10}$  &$011,111,100$    &$.693$     &$.774$             &1.242 \\
	$X_{11}$  &$111,100,100$     &$.693$    &$.763$              & 1.04  \\
	$A_1$       &$110,101,001$    &$.481$   &$.611$                       &$\infty$              \\
	$A_2$       &$110,011,010$    &$.481$   &$.575$                  &.962    \\
\end{tabular}
\medskip
\caption{Estimates of some tree shift entropies}
\end{table}
 \bigskip

\begin{example}
	For the matrix $M$ whose rows are $0 1 0, 0 0 1, 1 1 0$, we find numerically that the maximum eigenvalue is $\lambda \approx 0.2812$, the right and left eigenvectors are $r \approx (0.57, 0.75, 1)$ and $v \approx (0.75, 1.32, 1)$, so that the constant in Proposition \ref{prop:upper} is $c \approx 1.75$, and the (rigorous) upper estimate for $h$ is $U\approx 0.56$. 
	Thus for this example we have
	\be
	0< \htop(\Sigma_M) = \log \lambda \approx 0.28 \leq h \leq U  \approx 0.56 < \log 2.
	\en
(Numerically, using the recurrence, we find $h \approx 0.36$.)
\end{example}

\noindent
{\em Questions:} 

\noindent
1. What is the {\em maximal} possible complexity function of a dyadic tree all of whose infinite paths starting at the root are labeled by sequences from a fixed Sturmian system? (Does this concept even make sense?)

\noindent
2. Suppose that in a dyadic tree we label the infinite paths from the root with sequences from a fixed Sturmian system in the following lexicographic manner. 
Down the leftmost path put the lexicographically minimal sequence. 
Starting at the left edge, so long as there is no choice for the next symbol (the block completed is not right special), just copy the entry on the left edge at both nodes below. 
If the block arrived at allows two successor symbols, put a $0$ on the node below and to the left, a $1$ on the node below and to the right. 
Continue to apply this rule along all infinite paths from the root. 
We should arrive with the tree completely filled in with {\em all} sequences in the Sturmian system written along the uncountably many infinite paths from the root, in lexicographic order left to right, with the minimal sequence down the left edge and the maximal sequence down the right edge. 
The labeling is accomplished by following paths along the Hofbauer-Buzzi Markovian diagram of the Sturmian system (see \cite{CP}). 
What is the complexity function $p_\tau (n)$ of this labeled tree?

\noindent
3. What if we follow the preceding scheme to label the dyadic tree, except that whenever we have a choice we put $0$ and $1$ on the two nodes below randomly and independently of all other choices? 
What complexity function do we get with probability $1$? Could the labeled tree have positive entropy?

\noindent
4. Is there an example for which the limit defining $h$ (see Theorem \ref{prop:limit}) is not equal to the infimum?

\noindent
5. For each $n \geq 0$ denote by $\Phi_n$ the set of all allowable labelings of $\Delta_n$. 
For $n \geq 0, \phi \in \Phi_n$, and $ a \in A$, denote by $s_\phi (a)$ the number of the $2^n$ terminal nodes of $\Delta_n$ that have the label $a$.
For each $a \in A$ denote by $t_a$ the $a$'th row sum of $M$, that is, the number of outgoing arrows from the vertex $a$ in the graph of the shift of finite type defined by $M$.
Then 
\be
|\Phi_{n+1}| = \sum_{\phi \in \Phi_n} \prod_{a \in A} (t_a^2)^{s_\phi(a)}.
\en
Switching the order of summation might be a starting point for obtaining upper and lower estimates for $h$.

\noindent
6. Let us imagine that for most $\phi$ the distribution of the $s_\phi(a)$ is approximately given by the measure $\mu$ of maximal entropy for the subshift defined by $M$, so that $s_\phi(a) \sim 2^n \mu[a]$. 
Then
\be
|\Phi_{n+1}| \sim |\Phi_n| \prod_{a \in A} (t_a^2)^{2^n \mu [a]} \sim |\Phi_0| \prod_{a \in A} t_a^{(2^{n+1}-1)\mu[A]},
\en
so we might suppose that 
\be
\lim \frac{\log|\Phi_{n+1}|}{2^{n+2}} \sim U_m:= {\log\prod_{a \in A}t_a^{\mu[a]}}=\sum_{a \in A} \mu[a]\log t_a.
\en
The latter expression is the average of the row sums of the transition matrix $M$ using weights for the symbols given by the measure of maximal entropy for the one-dimensional shift of finite type. This may often be a good first approximation to $h$, usually from below.

\noindent
7. Measures of maximal entropy exist on tree shifts because the action is expansive, but determining uniqueness and identifying them, perhaps as Markov random fields (cf. \cites{Pavlov2012, MarcusPavlov2013}) appear to be significant problems. See \cite{MairesseMarcovici2017} for recent work on this problem.

\noindent
8. There is numerical evidence for the conjecture that the topological entropy $h^{(k)}$ of the tree shift on the regular $k$-ary tree defined by an irreducible $d \times d$ $0,1$ matrix $M$ with maximum row sum $s$ (1) is a strictly increasing function of $k$ (extending Theorem \ref{thm:htop}) and (2) has limit $\log s$. 

We now sketch a proof of (2). Let $a$ be a symbol in the alphabet which has $s$ successors allowed by the adjacency rules specified by $M$. 
We construct at least $s^{k^n}$ labelings of $\Delta_n$ as follows. First, assign to each of the $k^{n-1}$ sites at level $n-1$ the symbol $a$. 
Then assign to each site at level $n-2$ the same allowed predecessor of $a$, call it $b_{n-2}$. 
To each site at level $n-3$ assign the same allowed predecessor of $b_{n-2}$, call it $b_{n-3}$, etc. 
To each of the $k^n$ sites at level $n$ in $\Delta_n$ assign any of the $s$ allowed successors of $a$. 
In this way we form at least $s^{k^n}$ labelings of $\Delta_n$ which respect the adjacency rules and can be extended to valid labelings of the entire $k$-ary tree. Therefore 
\be
p(n) \geq s^{k^n} \quad \text{ and } \quad h^{(k)} =  \lim_{n \to \infty} \frac{\log p(n)}{(1+k+ \dots + k^n)} \geq \frac{k-1}{k} \log s.
\en
The reverse inequality follows from
\be
p(n) \leq d s^k s^{k^2} \dots s^{k^n} \quad \text{ for all } n.
\en

\begin{ack*}
	We thank Professors Jung-Chao Ban, Chih-Hung Chang, and Nic Ormes for discussions on this topic, Professor Olivier Carton and his colleagues for alerting us to the papers \cites{Carpi2001, Berstel2010}, and the referees for helpful comments..
\end{ack*}

 \end{document}